\documentclass[11pt,reqno]{amsart}
\usepackage{xifthen}

\usepackage{amsmath}
\usepackage{amssymb, latexsym, amsfonts, amscd, amsthm, mathrsfs}
\usepackage[usenames,dvipsnames]{color}
\usepackage[all]{xy}
\usepackage{graphicx}
\usepackage{epsfig}
\usepackage[colorlinks=true,linkcolor=blue]{hyperref}
\usepackage{amssymb}
\usepackage{enumerate}
\usepackage{bm}

\numberwithin{equation}{section}

\definecolor{purple}{rgb}{0.9,0,0.8}

\definecolor{gray}{rgb}{0.7,0.7,0.7}

\newtheorem{thm}{Theorem}[section]

\newtheorem{lem}[thm]{Lemma}
\newtheorem{prop}[thm]{Proposition}

\newtheorem{dfn}[thm]{Definition}

\theoremstyle{definition}

\newtheorem{rmk}[thm]{Remark}
\newtheorem*{rmk*}{Remark}

\newcommand{\beq}{\begin{equation}}
	\newcommand{\eeq}{\end{equation}}

\newcommand{\ra}{\rightarrow}

\renewcommand{\setminus}{\backslash}

\begin{document}

\title{$K_{r,s}$ GRAPH BOOTSTRAP PERCOLATION }

\author[E. \ Bayraktar]{Erhan Bayraktar} \thanks{E. Bayraktar is partially supported by the National Science Foundation under grant DMS-2106556 and by the Susan M. Smith chair.}
\address{Department of Mathematics, University of Michigan, Ann Arbor, U.S.A}
\email{erhan@umich.edu}
\author[S. \ Chakraborty]{Suman Chakraborty}
\address{Department of Mathematics and Computer Science,  Eindhoven University of Technology, Eindhoven, Netherlands}
\email{s.chakraborty1@tue.nl}

\date{\today}
\subjclass[2010]{Primary: 82B43, 05C80. }% Secondary: ;}
\keywords{Percolation, graph bootstrap percolation, random graphs, bipartite graphs}

\maketitle

\begin{abstract}
A graph $G$ percolates in the $K_{r,s}$-bootstrap process if we can add all missing edges of $G$ in some order such that each edge creates a new copy of $K_{r,s}$, where $K_{r,s}$ is the complete bipartite graph. We study $K_{r,s}$-bootstrap percolation on the Erd\H{o}s-R\'{e}nyi random graph, and determine the percolation threshold for balanced $K_{r,s}$ up to a logarithmic factor. This partially answers a question raised by Balogh, Bollob\'as, and Morris. We also establish a general lower bound of the percolation threshold for all $K_{r,s}$, with $r\geq s \geq 3$.
\end{abstract}
\section{Introduction}
For a given graph $H$, the $H$-bootstrap process is defined as follows. Let $G$ be a graph on vertex set $[n]:=\{1,2,\ldots,n\}$ and $K_n$ be the complete graph on the same set of vertices. Set $G_0 =G$ and define, for each $t\geq 0$,
$$
G_{t+1} := G_t \cup \big\{e \in E(K_n) : \exists  H \text{ with } e \in H \subset G_t \cup \{e\} \big\}.
$$

Let ${\langle G \rangle}_H = \cup_{t\geq 0} {G_t}$. Here ${\langle G \rangle}_H$ is the closure of $G$ under the $H$-bootstrap process. We say $G$ percolates under the $H$-bootstrap process on $K_n$ if ${\langle G \rangle}_H = K_n$. 
\par
Recently this process was studied by Balogh, Bollob\'as, and Morris for $G=G_{n,p}$, where $G_{n,p}$ is the random graph on $n$ vertices in which each edge is present independently with probability $p$. In \cite{balogh2012graph}, they defined the critical threshold for $H$-bootstrap percolation on $K_n$ as follows:
\[
p_c(n, H) := \inf\{p:P\left({\langle G \rangle}_H = K_n\right)\geq 1/2\}
\]
In this short article we study upper and lower bounds of $p_c(n, H)$ for $H=K_{r,s}$, where $K_{r,s}$ is the complete bipartite graph with $r$ vertices in one part and $s$ in the other. Here and throughout the paper we will assume $r\geq s \geq 3$ without loss of generality. Let 
\[
\lambda(r,s) : = \frac{rs-2}{r+s-2}.
\] The following theorem is the main result of this paper.

\begin{thm}\label{thm:sun9pmaprl}
Let $r\geq 4$, $s\geq 3$, and $s\leq r\leq (s-2)^2 +s$. Then there exist constants $c(r,s), C(r,s)>0$ such that for large enough $n$,
\begin{equation}
c(r,s) (\log n)^{-1} n^{-1/\lambda(r,s) } \leq p_c(n, K_{r,s}) \leq C(r,s) {\left(\frac{\log n}{\log\log n}\right)}^{2/\lambda(r,s)} n^{-{1/{\lambda(r,s)}}}.
\end{equation}
\end{thm}
\begin{rmk} This partially answers a question by Balogh, Bollob\'as, and Morris; see Problem 5 in \cite{balogh2012graph}. For the case $K_{2,t}$ some results have been recently obtained that we discuss in the next section. In the next proposition we obtain a general lower bound on $p_c(n, K_{r,s})$. One can also obtain a general lower bound using Proposition 25 in \cite{balogh2012graph}, but for $K_{r,s}$ the following  proposition provides a better lower bound.
\end{rmk}

\begin{prop}\label{prop:sun527pmapril}
For any $r,s\geq 3$,
 $$
 p_{c}(n,K_{r,s}) \geq (e \log n)^{-1}   {\lambda((s-2)^2 +s, s)}^2  n^{-1/\lambda((s-2)^2 +s, s) }. 
 $$
\end{prop}

\subsection{Related results} 
Graph bootstrap percolation is an example of cellular automata introduced by von Neumann \cite{von1966cell} (see also \cite{burks1969neumann}). Bollob\'as \cite{bollobas1968weakly} introduced $H$-bootstrap percolation, which is also known as weak saturation. Extremal questions are well studied when $H=K_r$ (see \cite{alon1985extremal}, \cite{frankl1982extremal}, and \cite{kalai1984weakly}). \par

More recently, graph bootstrap percolation has been studied on random graphs (see \cite{bollobas_2001} for an exposition on random graphs). In the context of the Erd\H{o}s-R\'{e}nyi random graph Balogh, Bollob\'as, and Morris \cite{balogh2012graph} obtained the following result regarding $K_r$ bootstrap percolation. It was shown that for $r\geq 4$, and $n \in \mathbb{N}$, sufficiently large
\[
\frac{n^{-1/\lambda(r)}}{2 e \log {n}} \leq p_c(n, K_r) \leq n^{-1/\lambda(r)} \log {n},
\]
where $\lambda(r) = \frac{{r \choose 2} -2}{r-2}$. Recently, extremal results have been studied for $H=K_{r,s}$, where $K_{r,s}$ is the complete bipartite graph with one part containing $r$ nodes, and $s$ nodes in the other. In \cite{gan2015ks}, the authors considered a related process called saturation. A graph $G$ is called called $H$ saturated if $G$ does not contain a copy of $H$, and adding any missing edge in $G$ completes a new copy of $H$. In \cite{gan2015ks}, it was shown that if $K_{n,n}$ is $K_{r,s}$ saturated then it must have at least $(r+s -2)n-(r+s -2)^2$ edges, confirming a conjecture in \cite{moshkovitz2015exact} up to an additive constant. In \cite{moshkovitz2015exact} the authors studied the weak saturation of $K_{r,s}$ in $K_{n,n}$, and showed that if it is $K_{r,s}$-weakly saturated in a bipartite graph, then it has at least $(2s - 2 + o(1))n$ edges, when $s\leq r$. Weak saturation of $K_{r,s}$ in $K_{n}$ has been studied in \cite{kronenberg2021weak}. In the context of random graph the authors in \cite{balogh2012graph} proposed the problem (Problem 5 in \cite{balogh2012graph} ) to determine $p_c(n, K_{r,s})$, at least up to a poly-logarithmic factor, for all $r,s \in \mathbb{N}$. It was shown in \cite{balogh2012graph} that 
\[
p_c(n, K_{2,3}) = \frac{\log{n}}{n} + \Theta\left(\frac{1}{n}\right).
\]
Recently some progress has been made for bipartite graphs of the form $K_{2,t}$. In \cite{bidgoli2018k_}, it was shown that 

\[
p_c(n, K_{2,4}) = \Theta\left(\frac{1}{n^{10/13}}\right).
\]
A lower and upper bound for $K_{2,t}$ is also obtained in \cite{bidgoli2018k_} for $t\geq 4$. Our result complements the results in \cite{bidgoli2018k_}, and determines $p_{c}(n,K_{r,s})$ up to poly-logarithmic factor when the graph is balanced (see Definition \ref{dfnbalanced930}). We also obtain a general lower bound for  $p_{c}(n,K_{r,s})$ when $r,s\geq 3$. 

\subsection{Remarks on the proof} Our proof of the lower bound in Theorem \ref{thm:sun9pmaprl}  is based on the witness set algorithm introduced in \cite{balogh2012graph}. The main idea involves two steps. The first step is to show that if a graph $G$ percolates under the $K_{r,s}$-bootstrap process on $K_n$ then there exists an witness set (see Section \ref{thusep724pm} for the precise definition) satisfying certain extremal properties. The second step is to show that if $p$ is below a certain threshold then there is no such set with high probability, that is, with probability going to one as the size of the graph goes to infinity. Although we use the same algorithm to establish the extremal properties of the witness set, the steps involved are different from those used in \cite{balogh2012graph} to prove bounds on the $K_r$-bootstrap process, and the analysis of the algorithm leads to a different optimization problem than in the case of $K_r$. Interestingly the condition required to establish the lower bound for $K_{r,s}$ using the witness set algorithm is also necessary to show that $K_{r,s}$ is balanced. Our lower bound works for $r,s\geq 3$. The upper bound directly uses Proposition 3 from \cite{balogh2012graph}. The assumptions in their proposition are valid when $r\geq4$ and $s\geq 3$. 

\section{Lower bound for $K_{r,s}$ percolation}\label{thusep724pm}
A novel Witness-Set algorithm was introduced in \cite{balogh2012graph} in the context of $K_r$-bootstrap percolation. We first fix some notations and then recall the algorithm for $H$-bootstrap percolation, for any finite graph $H$. We start with a graph $G$, and run the $H$-bootstrap process, that is we add the edges in ${\langle G \rangle}_H\setminus G$ one by one. The edges in ${\langle G \rangle}_H\setminus G$ are called infected edges. First, let us fix an ordering of the edges $(e_1,e_2,\ldots,e_k)$ in ${\langle G \rangle}_H\setminus G$. More precisely, $e_1$ is the first edge that was added (we also say `$e_1$ is infected') and $H^1$ is the copy of $H$ that was completed by adding $e_1$ (if more than one is completed we arbitrarily choose one), and continuing similarly for $i=2,3,\ldots,k$, let $e_i$ be the $i$-th edge that was added (or infected), and $H^i$ be the copy of $H$ that was completed by adding $e_i$ (if more than one is completed we arbitrarily choose one). We are now ready to state the  Witness-Set algorithm.
\subsection*{Witness-Set Algorithm} Assign a graph $F(e) \subseteq G$, to each edge $e \in {\langle G \rangle}_H$. The set of edges of $F(e)$, denoted by $E(F(e))$ is obtained as follows:  
\begin{itemize}
\item If $e \in G$ then set $E(F(e)) = \{e\}$. 

\item If $e=e_i$ for some $i=1,2,\ldots,k$ then
\beq\label{fri627feb}
E(F(e)) := \bigcup_{e' \neq e \in E(H^i)} E(F(e')).
\eeq
\end{itemize}
Now $F(e)$ is the graph whose vertices are the endpoints of the edges in $E(F(e))$, and edge set $E(F(e))$. The graph $F(e)$ is called the Witness-Set of edge $e$. Note that in \eqref{fri627feb} the union is taken only over the edges of $H$. In particular in the bootstrap process when a copy of $H$ is completed on the set of vertices, say, $V(H)$, there might be additional edges in the graph induced by $V(H)$, and the union is not taken over such edges.
\subsection*{The Red Edge Algorithm} Let $G$ be a graph, and $e \in {\langle G \rangle}_H\setminus G$.
\begin{itemize}
\item Run the Witness-Set Algorithm until the edge $e$ is infected.
\item Let $(e_{a_1}, e_{a_2}, \ldots, e_{a_m})$ be the infected edges which satisfy $F(e_{a_j}) \subseteq F(e)$, where $e_{a_m} = e$ and $a_1<a_2<\ldots<a_m$.
\item Call the set of edges $S_{e}:=\{e_{a_1},e_{a_2},\ldots,e_{a_m}\}$ red edges, and note that $e_{a_j} \in H^{a_j} \setminus (H^{a_1} \cup \ldots \cup H^{a_{j-1}})$. 
\end{itemize}

Therefore $F(e) = \left(H^{a_1} \cup \ldots \cup H^{a_m}\right)\setminus \{e_{a_1},e_{a_2},\ldots,e_{a_m}\}$. 

For an edge $ e \in {\langle G \rangle}_H \setminus G$ run the Red-Edge Algorithm, and let $S_{e} = \{e_{a_1},e_{a_2},\ldots,e_{a_m}\}$. Then for $t \in [m]$, define 
\[
B_t(e):= (H^{a_1} \cup \ldots \cup H^{a_t}) \setminus \{e_{a_1},\ldots, e_{a_t}\}.
\]
 Also, define a graph $\mathscr{G}_t(e)$, obtained using the Red Edge Algorithm whose vertices are the graphs $\{H^{a_1}, H^{a_2}, \ldots, H^{a_t}\}$, and in which two nodes $H^{a_i}$, and $H^{a_j}$ are adjacent if they share at least one common edge. \par
 Let us now make few remarks about the Red Edge algorithm. First note that, if $e=e_i$ for some $i=1,2,\ldots,k$, then $F(e)$ can be interpreted as the subset of $G$ that causes the infection of $e$. For each $j=1,2,\ldots,m$, the condition $F(e_{a_j})\subseteq F(e)$  implies that $H^{a_j}\setminus \{e_{a_1},e_{a_2},\ldots,e_{a_m}\}\subseteq F(f)$ for some $f \in H^{a_m}$. In words, at the $j$-th step in the Red Edge algorithm $e_{a_j}$ is added and $H^{a_j}$ is completed. The condition $F(e_{a_j})\subseteq F(e)$ ensures the graph $\mathscr{G}_m$ is connected (this will be useful in our proof; see Lemma \ref{feb243pm}, Lemma \ref{feb421pm} and Lemma \ref{feb242pmsun} below for more details). In Lemma \ref{feb421pm} below we obtain an upper bound on the number of edges in $B_t$, which will be used to obtain the lower bound in Theorem \ref{thm:sun9pmaprl}.    \par
 We now provide an explicit example to illustrate the Witness-set algorithm and the Red Edge algorithm in Figure 1. In this figure, the nodes of $G$ are given by the set $\{A,B,C,D,E,F,G,H\}$, and the edges of $G$ are drawn in black. The edges of ${\langle G \rangle}_H\setminus G$ are drawn in red, where $H$ is a triangle (complete graph on three vertices). We first ordered the edges in ${\langle G \rangle}_H\setminus G$ and marked them by $e_1,e_2,\ldots,e_7$, and the triangle completed by adding them are $ABC,ACD,BCD,FGH,ADE,BAE$, and $CDE$ respectively. Now let us run the Witness-set algorithm until the edge $e=e_5$ gets infected, we get
\begin{align*}
	&F(e_1)=\{AB,BC\}, \,\,	F(e_2)=\{AB,BC,CD\}, \,\, F(e_3)=\{BC,CD\}, \,\,  F(e_4)=\{GF,GH\},\\& F(e_5)=\{DE,AB,BC,CD\}.
\end{align*}
Now in the Red Edge algorithm, for $e=e_5$ we have $S_{e} = \{e_1,e_2,e_3,e_5\}$ and 
\begin{align*}
B_1(e_5) = \{AB,BC\}, B_2(e_5)= \{AB,BC,CD\}, B_3(e_5)=\{AB,BC,CD\}, B_4(e_5)=\{AB,BC,CD,DE\}.
\end{align*}
\begin{figure}\label{fig:451pm01apr21}
	\includegraphics[width=14cm]{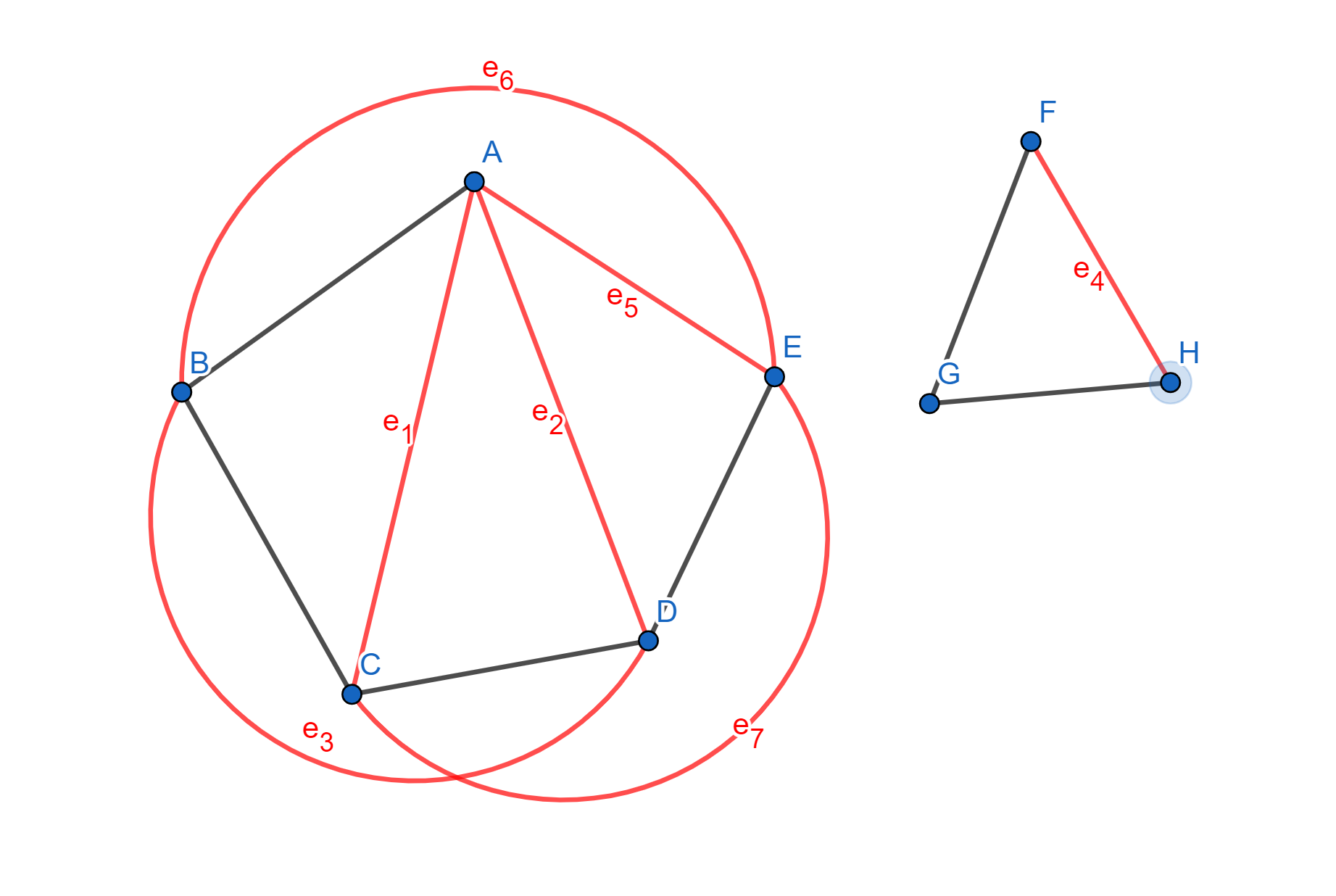}
	\caption{Example of the running of the Witness-set algorithm and the Red Edge algorithm}
	\centering
\end{figure}
Here note, for example, that $B_3(e_5) \neq F(e_3)$, and $B_1(e_3)= F(e_3)= \{BC,CD\}$. Also, quantities such as $B_3(e_3)$ does not make sense. \par
Let us now start with two basic results. For an edge $ e \in {\langle G \rangle}_H \setminus G$ run the Red-Edge Algorithm, and let $S_{e} = \{e_{a_1},e_{a_2},\ldots,e_{a_m}\}$. Then consider the graph $\mathscr{G}_m(e)$,  whose vertices are the graphs $\{H^{a_1}, H^{a_2}, \ldots, H^{a_m}\}$, and in which two nodes $H^{a_i}$, and $H^{a_j}$ are adjacent if $E(H^{a_i})\cap E(H^{a_j})$ is non-empty. With this notation, the first one (Lemma \ref{feb243pm}) states that the graph $\mathscr{G}_m(e)$ is connected, and the second one (Lemma \ref{sat439pmfeb}) ensures that all witness sets can not be very large. 
\begin{lem}\label{feb243pm}
Let $F(e)$ be a Witness-Set for $e$ on the graph $G$. Then $\mathscr{G}_m(e)$ is a connected graph.
\end{lem}
\begin{proof}[Proof of Lemma \ref{feb243pm}] Take $f \in F(e)$. Then we claim that there is a path in $\mathscr{G}_m(e)$ from $H^{a_m}$ to $H^{a_t}$ for some $t\in [m]$, where $f\in E(H^{a_t})$. Indeed, since $f \in F(e)$ either $f \in E(H^{a_m})$, in which case we are done or there is an $f_1\in E(H^{a_m})$ such that $f$ belongs to the Witness-Set of $f_1$.  Let $H^{(1)}$ be the copy of $H$ that was created by the Red-Edge algorithm by the addition of $f_1$ (note that $H^{(1)}$ and $H^{a_m}$ must be two different copies of $H$). Clearly, $H^{a_m}$ and $H^{(1)}$ are adjacent in $\mathscr{G}_m(e)$. Then again either  $f \in E(H^{(1)})$, in which case we are done or there is an $f_2\in E(H^{(1)})$ such that $f$ belongs to the Witness-Set of $f_2$.  Let $H^{(2)}$ be the copy of $H$ that was created by the Red-Edge algorithm by the addition of $f_2$ (note that $H^{(2)}, H^{(1)}$, and $H^{(a_m)}$ must be three different copies of $H$). Again, $H^{(2)}$ and $H^{(1)}$ are adjacent in $\mathscr{G}_m(e)$. Continuing this similarly the claim is proved since there are only $m$ distinct copies of $H$ that were created by the Red-Edge algorithm. \par 
Now for $j \in [m]$, $F(e_{a_j}) \subseteq F(e_{a_m})$. Since the set $F(e_{a_j})$ is non-empty, there exists an edge $f \in F(e_{a_j}) \cap F(e)$. Thus there is a $t \in [m]$, such that there is a path in $\mathscr{G}_m$ from  $H^{a_m}$ to $H^{a_t}$ and $f \in E(H^{a_t})$.  Also since $f \in F(e_{a_j})$, there is a path from $H^{a_j}$ to $H^{a_{t'}}$ such that $f \in E(H^{a_{t'}})$. These give $f \in E(H^{a_t})\cap E(H^{a_{t'}})$. Therefore either $t = t'$ or $H^{a_t}$ and $H^{a_{t'}}$ are neighbors. Thus there is a path from $H^{a_m}$ to $H^{a_j}$.
\end{proof}
\begin{rmk}
	It is not difficult to see that $\mathscr{G}_t(e)$ is not necessarily a connected graph for all $t \in [m]$. Nevertheless, we will only use the fact that $\mathscr{G}_m(e)$ is connected to deduce Lemma \ref{feb242pmsun} from Lemma \ref{feb421pm}.
\end{rmk}

\begin{lem}\label{sat439pmfeb}
Let $F(e)$ be an Witness-Set for $e$ on the graph $G$. Let $L \in \mathbb{N}$. If $e(F(e)) \geq L$, then there exists an edge $f \in E\left({\langle G \rangle}_H\right)$ with
\beq\label{451pmfebsat}
L \leq e(F(f)) \leq e(H) L
\eeq
in the same realization of the Witness-Set algorithm. 
\end{lem}
\begin{proof}[Proof of Lemma \ref{sat439pmfeb}]
Firstly, if $e(F(e)) \leq e(H)L$ then we can take $f=e$, and we are done. Otherwise, consider an instance of the Witness-Set algorithm when $e_1,e_2,\ldots, e_l$ are already infected, and after that $f$ is next in line to be infected. Then by \eqref{fri627feb} 
 $$
 e(F(f)) \leq e(H) \max_{1\leq i \leq l} e(F(e_i)).
 $$
 In other words $ e(F(f)) \leq e(H) e(F(e_i))$ for $1\leq i \leq l$. Therefore if $e(F(e)) > e(H)L$ one witness set satisfying \eqref{451pmfebsat} must be created in the process with $F(f) \subset F(e)$. 
\end{proof}

The following lemma provides us the key estimate to establish the lower bound. Let us fix some notations before stating the lemma. Let $l_t$ denote the number of components of $\mathscr{G}_t(e)$. Also let $c_t(v)$ denote the number of components of $\mathscr{G}_t(e)$ containing the vertex $v \in V(G)$, and define
$$
k_t := \sum_{v \in V(B_t)} {\left(c_t(v)-1 \right)}.
$$
We are now ready to state the lemma and the proof is deferred to the end of this section.
\begin{lem}\label{feb421pm}
For $r \geq 3$, $s\geq 3$, and $r\leq (s-2)^2 +s$ we have $$e(B_t(e)) \geq \frac{rs-2}{r+s-2} \left(v(B_{t}(e)) + k_{t} - l_{t}(r+s))\right) + l_t (rs-1).$$
\end{lem}

Note that Lemma \ref{feb243pm} gives $\mathscr{G}_m(e)$ is a connected graph, and hence $l_m=1$, and $k_m = 0$. Now the following lemma is immediate from Lemma \ref{feb421pm}. 
\begin{lem}\label{feb242pmsun}
Recall that $\lambda(r,s) = \frac{rs-2}{r+s-2}$, under the conditions of Lemma \ref{feb421pm} we have
\beq\label{eqsun242feb}
e(F(e)) \geq \lambda(r,s) \left(v(F(e)) - 2\right) + 1.
\eeq
\end{lem}
Next we will show that the expected number of witness sets is asymptotically negligible when $p$ is smaller than certain threshold. For each $m \in \mathbb{N}$ and every $e \in E(K_n)$, define
\[
 Y_m(e) = |\{F \subset G_{n,p}: e \subset V(F),\text{ and } e(F) = m\geq \lambda(r,s)\left(v(F) - 2\right) + 1\}|.
 \]
 
  Here, $Y_m(e)$ counts the number of subgraphs $F$ of $G_{n,p}$ whose vertex set contains the end points of edge $e$, and has $m \geq \lambda(r,s)\left(v(F) - 2\right) + 1$ edges.
 
\begin{lem}\label{lem:thu207feb} Let $r,s \geq 3$, $e p n^{1/\lambda(r,s) } (\log{n}) {rs}\leq {\lambda(r,s)}^2$. Then there exists a constant $C(r,s)$ such that for sufficiently large $n$,
\[
\mathbb{E}(Y_m(e)) \leq \frac{C(r,s)}{n^{1/\lambda(r,s)}} \left(\frac{m+C(r,s)}{2rs\log{n}}\right)^{m-\lambda(r,s)-1},
\]
for $ m\leq rs \log {n}$.
\end{lem}  
\begin{proof}[Proof of Lemma \ref{lem:thu207feb}]
 Now for a fixed $m$, let $l \in \mathbb{N}$ be maximal such that $m\geq \lambda(r,s)\left(l - 2\right) + 1$. Therefore $v(F)\leq l$, and hence
 $$
 \mathbb{E}(Y_m(e)) \leq \sum_{j=0}^{l-2} {n \choose j} {{j\choose 2} \choose m}p^{m} \leq \sum_{j=0}^{l-2} {n \choose j} {{l \choose 2} \choose m}p^{m} \leq \sum_{j=0}^{l-2} {n \choose j} {{l^2/2} \choose m}p^{m} \leq {2 n^{l-2}}{\left(\frac{epl^2}{2m}\right)^m}. 
 $$
 
 The first inequality is a crude upper bound of number of possible $F$ that has at most $l$ vertices and exactly $m$ edges. Note that the sum runs up to $l-2$ since $F$ must contain the endpoints of the edge $e$. In the last inequality we used 
 $$ \sum_{j=0}^{l-2} {n \choose j} \leq \sum_{j=0}^{l-2} {n^ j} \leq n^{l-2}\left(1+\frac{1}{n}+\frac{1}{n^2}+\ldots \right) \leq 2n^{l-2}.$$ 
  Since $m\leq rs \log {n}$, for a sufficiently large constant ${C}(r,s)$ we have  $l \leq \frac{m-1}{\lambda(r,s)} +2 \leq {C}(r,s) \log {n}$. Therefore $p l^2 = o(1)$, and consequently $\frac{epl^2}{2m} =o(1)$. We also have $m \geq \lambda(r,s)\left(l - 2\right) + 1$. Combining these and the last display we have
 $$
 \mathbb{E}(Y_m(e)) \leq {2 n^{l-2}} {\left(\frac{epl^2}{2m}\right)^m} \leq 2 \left(n^{1/\lambda(r,s)} \frac{epl^2}{2m}\right)^{\lambda(r,s) \left(l - 2\right)}{\left(\frac{epl^2}{2m}\right)} . 
 $$
 
 Let $C(r,s)>0$ be large enough such that $m(m+C(r,s)) \geq (m+2\lambda(r,s))^2\geq (\lambda(r,s) l)^2$.  Then we will have
 \[
 \frac{l^2}{m} \leq \frac{m+C(r,s)}{{\lambda(r,s)}^2}.
 \]
 Hence using $m\leq rs \log{n}$, we have (enlarging the constant $C(r,s)$ if needed)
\[ \frac{ep l^2}{m} \leq \frac{C(r,s)}{n^{1/\lambda(r,s)}}  \]
 Using the last three displays we immediately have
 \begin{align*}
 \mathbb{E}(Y_m(e)) &\leq \frac{C(r,s)}{n^{1/\lambda(r,s)}} \left(n^{1/\lambda(r,s)} \frac{ep({m+C(r,s)})}{2 {\lambda(r,s)}^2}\right)^{\lambda(r,s) \left(l - 2\right)} \\
 &\leq \frac{C(r,s)}{n^{1/\lambda(r,s)}} \left(\frac{m+C(r,s)}{2rs\log{n}}\right)^{\lambda(r,s) \left(l - 2\right)}.
\end{align*}
Finally using the facts $m\leq rs \log{n}$, and the fact that $\lambda(r,s)\left(l - 2\right) > m-\lambda(r,s)-1$ (since $l$ is maximal such that $m\geq \lambda(r,s)\left(l - 2\right) + 1$) we have
\[
 \mathbb{E}(Y_m(e)) \leq \frac{C(r,s)}{n^{1/\lambda(r,s)}} \left(\frac{m+C(r,s)}{2rs\log{n}}\right)^{m-\lambda(r,s)-1}.
\]
\end{proof}
\begin{prop}\label{prop:thu217feb}
Let $r,s\geq 3$, and $e \in E(K_n)$. If $e p n^{1/\lambda(r,s) } (\log{n}) {rs}\leq {\lambda(r,s)}^2$ then 
\[
\mathbb{P}\left(e \in  {\langle G_{n,p} \rangle}_{K_{r,s}}\right) \to 0
\]
as $n \to \infty$.
\end{prop}
\begin{proof}[Proof of Proposition \ref{prop:thu217feb}]
We shall prove that, for every $e \in E(K_n)$,
\[
 \mathbb{P}\left(e \in {\langle G_{n,p} \rangle}_{K_{r,s}}\right) \rightarrow 0
\]
as $n \ra \infty$. For an edge $e \in {\langle G_{n,p} \rangle}_{K_{r,s}} $,  consider the set $F=F(e) \subset G_{n,p}$, obtained using the Witness-Set Algorithm. There can be two cases: in the first case $e \in G_{n,p}$, or in the second case (when $e \notin G_{n,p}$), we must have $e \in {\langle G_{n,p} \rangle}_{K_{r,s}} \setminus G_{n,p}$ using Lemma \ref{feb242pmsun} we will get an Witness-Set $F=F(e)$ such that
\[
e(F) \geq \lambda(r,s) \left(v(F) - 2\right) + 1.
\]
Now let us assume $e(F) \leq \log{n}$ for the time being. 
 In the second case we must have $Y_{m}(e) \geq 1$ for some 
\[
\lambda(r,s) \left(r+s - 2\right) + 1\leq m \leq \log {n}.
\]  
Using Markov's inequality and Lemma \ref{lem:thu207feb} we have the probability that  $e \in {\langle G_{n,p} \rangle}_{K_{r,s}}$ is at most
\[
p +\sum_{m= \lambda(r,s) \left(r+s - 2\right) +1}^{\log {n}} \mathbb{E}(Y_m(e)) \leq p +\frac{C(r,s)}{n^{1/\lambda(r,s)}} \sum_{m= \lambda(r,s) \left(r+s - 2\right) +1}^{\log {n}} \left(\frac{m+C(r,s)}{2rs\log{n}}\right)^{m-\lambda(r,s)-1}.
\]
Since $r+s\geq 6$ and $\lambda(r,s)>1$ we have each term in the last sum going to zero as $n \ra \infty$ and there are at most $\log {n}$ terms. Hence the factor $n^{-1/\lambda(r,s)}$ ensures the whole term in the last display goes to zero as $n \to \infty$. \par 
We are now left with the part $e(F)> \log {n}$. In this case Lemma \ref{sat439pmfeb}, gives that there must be an edge $f$ in $K_n$ such that $\log {n} \leq e(F(f)) \leq rs \log {n}$. Therefore $Y_m(f) \geq 1$ for some $\log {n} \leq m \leq rs \log {n}$ Now using Lemma \ref{lem:thu207feb}, the union bound, and Markov's inequality the probability that such an edge exists is at most
\[
{n \choose 2} \frac{C(r,s)}{n^{1/\lambda(r,s)}} \sum_{m=\log{n}}^{rs \log {n}} \left(\frac{m+C(r,s)}{2rs\log{n}} \right)^{m-\lambda(r,s) -1} \leq n^2 \frac{C(r,s)}{n^{1/\lambda(r,s)}} \left(\frac{2}{rs}\right)^{\log{n} - \lambda(r,s)-1}.
\]
Here we use the fact that the function $\left(\frac{x+C(r,s)}{2rs\log{n}} \right)^{x}$ is decreasing in the interval $[\log{n}, rs\log{n}]$ for sufficiently large $n$ and $rs\geq 9$. To complete the proof we again use $rs\geq 9$, and ensure that the last display is converging to zero as $n\to \infty$.
\end{proof}
Now we will prove Lemma \ref{feb421pm}, which will complete the proof of the lower bound. 
\begin{proof}[Proof of Lemma \ref{feb421pm}] To simplify the notation, in this proof we will write $B_{t}(e)=B_t$, and $\mathscr{G}_{t}(e)=\mathscr{G}_{t}$. We will do induction on $t$. If $t=1$ then $v(B_1) = r+s$, $e(B_1)=rs-1$, $l_1=1$, and $k_1=0$. Therefore
\[
e(B_1) = rs-1 = \frac{rs-2}{r+s-2} \left(v(B_{1}) + k_{1} - l_{1}(r+s))\right) + l_1 (rs-1).
\]
Thus the lemma holds for $t=1$. For $t\geq 2$ we break it down into three cases.
\subsection*{Case I} $l_t = l_{t-1} +1$. \par
In this case all edges of $K_{r,s}^t$ are new (an edge $e \in K_{r,s}^t$ is called new when $e \notin K_{r,s}^{i}$ for $i=1,2,\ldots,t-1$). Indeed, otherwise if it shares an edge with already existing edges then it must belong to one of the components of $\mathscr{G}_{t-1}$ but $\mathscr{G}_{t}$ has one more component than $\mathscr{G}_{t-1}$. Therefore
\[
e(B_t) = e(B_{t-1}) +rs-1.
\]

Let $b$ be the number of vertices of $K_{r,s}^t$ that are not new.  Hence $v(B_t) = v(B_{t-1}) +r+s-b$ and $k_t = k_{t-1} +b$ (these $b$ vertices are in one more component in $\mathscr{G}_{t}$ than in $\mathscr{G}_{t-1}$). Let us now use these and the induction hypothesis for $t-1$ to get 
\begin{align*}
e(B_t) &\geq \left(\frac{rs-2}{r+s-2}\right)\left(v(B_{t-1}) + k_{t-1} - l_{t-1}(r+s))\right) + l_{t-1} (rs-1) +rs-1 \\
&= \left(\frac{rs-2}{r+s-2}\right)\left(v(B_{t}) -r-s+b + k_{t} -b - (l_{t} -1)(r+s))\right) +(l_{t-1}+1) (rs-1) \\
&= \left(\frac{rs-2}{r+s-2}\right)\left(v(B_{t}) + k_{t}  - l_{t} (r+s))\right) +l_{t} (rs-1).
\end{align*}
Therefore the lemma is proved for $l_t = l_{t-1} +1$.
\subsection*{Case II} $l_t = l_{t-1}$. \par
In this case $K_{r,s}^t$ shares at least one edge with some component $C_1$ of $\mathscr{G}_{t-1}$. Also all other edges that are not shared with $C_1$ must be new. Let $b$ be the number of vertices of $K_{r,s}^t\setminus C_1$ which are not new and $a$ be the number of vertices in $K_{r,s} \cap C_1$. We have the following inequality now
\[
e(B_t) \geq e(B_{t-1}) +rs-1 - |\{\text{edges shared with } C_1\}|
\]
Using the induction hypothesis 
\[
e(B_t) \geq \left(\frac{rs-2}{r+s-2}\right)\left(v(B_{t-1}) + k_{t-1} - l_{t-1}(r+s))\right) + l_{t-1} (rs-1) +rs-1 - |\{\text{edges shared with } C_1\}|
\]
Here note that $v(B_t) = v(B_{t-1})+r+s-a-b$, $k_t = k_{t-1} +b$, Thus
\[
e(B_t) \geq \left(\frac{rs-2}{r+s-2}\right)\left(v(B_{t})-r-s+a + k_{t} - l_{t}(r+s))\right) + l_{t} (rs-1) +rs-1 - |\{\text{edges shared with } C_1\}|
\]
Finally we have 
\begin{align*}
e(B_t) \geq \left(\frac{rs-2}{r+s-2}\right)&\left(v(B_{t}) + k_{t} - l_{t}(r+s))\right)+ l_{t} (rs-1)\\
& +(a-r-s)\left(\frac{rs-2}{r+s-2}\right) +rs-1 - |\{\text{edges shared with } C_1\}|.
\end{align*}
Therefore we will be done if we show that 
\beq\label{eqn:fri151feb}
(a-r-s)\left(\frac{rs-2}{r+s-2}\right) +rs-1 - |\{\text{edges shared with } C_1\}| \geq 0.
\eeq
Divide the vertices of $K_{r,s}$ into an $r$-subset, and an $s$-subset, such that each vertices one of these subset has an edge to every vertices of the other subset. Now we denote $|K_{r,s}^t \cap C_1|=a$, and let $K_{r,s}^t \cap C_1$ consists of $P$ vertices from the $r$-subset and $Q$ vertices from the $s$-subset. Therefore $|\{\text{edges shared with } C_1\}| = PQ$. Since at least one edge is shared $1\leq P \leq r$, and $1\leq Q \leq s$. Therefore showing \eqref{eqn:fri151feb} reduces to the problem to showing 
\beq\label{eqn:fri203feb}
(P+Q-r-s)\left(\frac{rs-2}{r+s-2}\right) +rs-1 - PQ \geq 0
\eeq
subject to the conditions $1\leq P \leq r$, and $1\leq Q \leq s$. Let us prove this with the additional constraint $1\leq P +Q \leq r+s-1$. Note that if we want to show \eqref{eqn:fri203feb} for a fixed $P$, then it is sufficient to check this for the endpoints i.e. $Q=1$ and $Q=s$ (since for a fixed $P$ \eqref{eqn:fri203feb} is linear in $Q$). Therefore we must check for each $1\leq P \leq r-1$ the following hold
\[
(P-1) \left(\frac{rs-2}{r+s-2}\right) +1- P \geq 0, \text{  and  } (P+s-2)\left(\frac{rs-2}{r+s-2}\right) +1-sP\geq 0.
\]
Again both these equations are linear in $P$ and therefore we check these equations for $P=1, r-1$. For $P= 1$ the first one trivially holds, and the second one is
\[
 (s-1)\left(\frac{rs-2}{r+s-2}\right) +1-s = (s-1)\left(\frac{rs-2}{r+s-2} -1\right)
 \]
 It is easy to show that $\frac{rs-2}{r+s-2}$ is non-decreasing both in $r$ and $s$ and therefore the last expression is non-negative as long as $r,s \geq 2$.
Similarly, for $P=r-1$ the first one 
\[
(r-2) \left(\frac{rs-2}{r+s-2}\right) +1- (r-1) = (r-2)\left(\frac{rs-2}{r+s-2} - 1\right) \geq 0.
\]
 For $P=r-1$ the second one boils down to the condition
\[
(r+s-3)\left(\frac{rs-2}{r+s-2}\right) +1-s(r-1) = \left(\frac{(s-2)^2+s -r}{r+s-2}\right) \geq 0.
\]
Therefore  we have shown \eqref{eqn:fri203feb} for $1\leq P\leq r-1$, and $1 \leq Q\leq s$. Let us check this for $P= r$ and $Q= s-1$. Indeed, since $r\geq s$, we have 
\[
(r+s-1-r-s)\left(\frac{rs-2}{r+s-2}\right) +rs-1 - r(s-1) = \frac{(r-2)^2 +r-s}{r+s-2} \geq 0. 
\]
The proof of is complete as long as $|K_{r,s}^t \cap C_1| \leq r+s-1$. Finally if $|K_{r,s}^t \cap C_1| = r+s$, then no new vertices were added by addition of $K_{r,s}^t$ and $v(B_t) = v(B_t-1)$. Therefore $e(B_t) \geq e(B_{t-1})$. In this case we also have $|K_{r,s}^t \cap C_1^c|=0$ and hence $k_t=k_{t-1}$. Hence the proof is complete when $l_t =l_{t-1}$.
\subsection*{Case III} $l_t < l_{t-1}$. \par
Let $m = l_{t-1}-l_{t}+1$. $K_{r,s}^t$ shares at least one edge with edge with the components $C_1, C_2,\ldots,C_m$, and it does not share any edge with the other components. Note that $m\geq 2$. Therefore
\[
e(B_t) \geq e(B_{t-1}) +rs-1-\sum_{i=1}^m|\{\text{edges shared with } C_i\}|.
\]
Let $\mathscr{P}(m)$ denote the set of all subsets of $\{1,2,\ldots,m\}$ and for $S \in \mathscr{P}(m)$ define
\[
a_S = |\{v \in K_{r,s}^t: v \in C_j \Leftrightarrow j \in S\}|.
\]
Here $a_S$ counts the number of nodes of $K_{r,s}^t$ that are in $C_j$ for $j \in S$, and are not in any other component.
Let $a = |K_{r,s}^t \cap \{C_1 \cup\ldots\cup C_m \}|$, and $b =$ number of vertices in $K_{r,s}^t\setminus \{C_1 \cup\ldots\cup C_m \}$ which are not new. Here $v(B_t)=v(B_{t-1})+r+s-a-b$. Since $K_{r,s}^t$ merges the components $\{C_1 \cup\ldots\cup C_m \}$ we have if $S=\{j \in [m]: v \in C_j\}$ then $c_t(v) = c_{t-1}(v)-|S|+1$. Therefore $k_t \leq k_{t-1}+b-c$ where $c = \sum_{S \in \mathscr{P}(m)}a_S\left(|S|-1\right)$. \par
Now we have
\begin{align*}
e(B_t) &\geq e(B_{t-1}) +rs-1-\sum_{i=1}^m|\{\text{edges shared with } C_i\}| \\
&\geq \left(\frac{rs-2}{r+s-2}\right)\left(v(B_{t-1}) + k_{t-1} - l_{t-1}(r+s))\right) + l_{t-1} (rs-1)\\
& +rs-1-\sum_{i=1}^m|\{\text{edges shared with } C_i\}|
\end{align*}
Plugging in the estimates we get
\begin{align*}
e(B_t)&\geq \left(\frac{rs-2}{r+s-2}\right)\left(v(B_{t}) + k_{t} - l_{t}(r+s))\right) + l_{t} (rs-1)\\
& + \left(\frac{rs-2}{r+s-2}\right)\left(a+c-2m\right)+m -\sum_{i=1}^m|\{\text{edges shared with } C_i\}|.
\end{align*}
Therefore we will be done if we show 
\[
\left(\frac{rs-2}{r+s-2}\right)\left(a+c-2m\right)+m \geq \sum_{i=1}^m|\{\text{edges shared with } C_i\}|.
\]
Now let us note that $a =  \sum_{S \in \mathscr{P}(m)}{a_S}$, and hence $a+c = \sum_{S \in \mathscr{P}(m)}a_S|S|$.
Therefore we will have to prove
\begin{align}\label{eqn:feb244sat}
\sum_{i=1}^m|\{\text{edges shared with } C_i\}| \leq \left(\frac{rs-2}{r+s-2}\right)\left(\sum_{S \in \mathscr{P}(m)}a_S|S|-2m\right)+m 
\end{align}

Note that $|K_{r,s}^t \cap C_j |= \sum_{\{S \in \mathscr{P}(m): S\ni j\}}a_S$, and consequently we have the following simple but important identity
\[
\sum_{j=1}^m{|K_{r,s}^t \cap C_i |} = \sum_{i=1}^m{\sum_{\{S \in \mathscr{P}(m): S\ni i\}}a_S}= \sum_{S \in \mathscr{P}(m)}a_S|S|.
\]
Now the inequality \eqref{eqn:feb244sat} becomes equivalent to the following
\[
\sum_{i=1}^m|\{\text{edges shared with } C_i\}| \leq \left(\frac{rs-2}{r+s-2}\right)\left(\sum_{i=1}^m{|K_{r,s}^t \cap C_i |}-2m\right)+m
\]

Next step is to turn it into an optimization problem. Let $K_{r,s}^t \cap C_i$ consists of $P_i$ vertices from the $r$-subset and $Q_i$ vertices from the $s$-subset of $K_{r,s}^t$ for $i=1,2,\ldots,m$. We will be done if we prove the following for $ P_i\geq 1$, $Q_i\geq 1$, and $\sum_{i=1}^m P_i \leq r$, $\sum_{i=1}^m Q_i \leq s$.

\[
\sum_{i=1}^m P_iQ_i  \leq \left(\frac{rs-2}{r+s-2}\right)\left(\sum_{i=1}^m{(P_i+Q_i)}-2m\right)+m,
\]
Lemma \ref{lem:sat6o3feb} gives that this is true with an additional constraint $\sum_{i=1}^m (P_i+Q_i) \leq r+s-1$. The only case remains when $\sum_{i=1}^m (P_i+Q_i) = r+s$ . In this case left side of the equation is at most $rs-m$ To see this,

 We have $\sum_{i=1}^mP_i=r,\sum_{i=1}^mQ_i = s$, and using Cauchy-Schwarz $(\sum_{i=1}^m P_iQ_i)^2 \leq \sum_{i=1}^m P_i^2 \sum_{i=1}^m Q_i^2$. Also $\sum_{i=1}^m P_iQ_i$ will be maximized if  $P_i =C Q_i$ and therefore $C=r/s$. Hence the maximum value is 
\begin{align*}
\sum_{i=1}^m P_iQ_i &= \frac{r}{s}\sum_{i=1}^mQ_i^2 = \frac{r}{s}\left((\sum_{i=1}^mQ_i)^2 - \sum_{i=1}^m\sum_{j=1, j \neq i}^mQ_iQ_j\right) \leq \frac{r}{s}\left(s^2 - \sum_{i=1}^m\sum_{j=1, j \neq i}^mQ_j\right) \\
&=\frac{r}{s}\left(s^2 - \sum_{i=1}^m(s-Q_i)\right) = \frac{r}{s}\left(s^2 - ms+s\right) = rs-rm+r.
\end{align*}
Now $rs-rm+r\leq rs-m$ iff $m \geq r/{r-1}$. Which is trivially true since $m\geq 2$. Therefore the right side is
\begin{align*}
\left(\frac{rs-2}{r+s-2}\right)\left(r+s-2m\right)+m &= rs-2+(2-2m)\left(\frac{rs-2}{r+s-2}\right) +m \\
&\geq rs-2+(2-2m) +m =rs-m,
\end{align*}
and the proof therefore is complete.
\end{proof} 
Let us prove the following technical Lemma that we have used in the proof of Lemma \ref{lem:thu207feb}.
\begin{lem}\label{lem:sat6o3feb}
Let $m\geq2$ and $3 \leq s\leq r$, $ P_i\geq 1$, $Q_i\geq 1$, and $\sum_{i=1}^m P_i \leq r$, $\sum_{i=1}^m Q_i \leq s$,  $\sum_{i=1}^m (P_i+Q_i) \leq r+s-1$, then
\[
\sum_{i=1}^m P_iQ_i  \leq \left(\frac{rs-2}{r+s-2}\right)\left(\sum_{i=1}^m{(P_i+Q_i)}-2m\right)+m,
\]

\end{lem}
\begin{proof}[Proof of Lemma \ref{lem:sat6o3feb}]
We will use induction on $m$. For $m=2$, we need to show
\[
P_1Q_1+P_2Q_2\leq \left(\frac{rs-2}{r+s-2}\right)(P_1+Q_1+P_2+Q_2 - 4)+2.
\]
If we fix any $3$ of $(P_1, Q_1, P_2, Q_2)$ then it becomes linear in the remaining variables. Therefore it is sufficient to verify this for the endpoints: $(1,1,r-1,1), (1,1,1,s-1), (1,1,1,1), (r-1,1,1,s-2), (r-2,1,1,s-1), (r-1,s-2,1,1), (r-2,s-1,1,1)$. \par
It trivially holds for $(1,1,1,1)$. For $(1,1,1,s-1), (1,1,r-1,1)$, it is easy using the fact that for $r,s\geq3$,  $\frac{rs-2}{r+s-2}\geq 1$.For $(1,1,r-1,s-2)$, (using $r,s\geq 3$, and $2x^2-7x$ is increasing for $x\geq 7/4$ )
\[
\left(\frac{rs-2}{r+s-2}\right)(r+s-1 - 4)+2 - 1-(r-1)(s-2) = \frac{2r^2-7r+s^2-5s+12}{r+s-2} \geq \frac{3(s-2)^2}{(r+s-2)}\geq0.
\]
For $(r-2,s-1,1,1)$ we again use $r,s\geq 3$, and $x^2-5x$ is increasing for $x\geq 5/2$,
\[
\left(\frac{rs-2}{r+s-2}\right)(r+s-1 - 4)+2 - 1-(r-1)(s-2) = \frac{2s^2-7s+r^2-5r+12}{r+s-2} \geq \frac{3(s-2)^2}{(r+s-2)}\geq0.
\] 
Finally for both the cases $(r-1,1,1,s-2)$, and $(r-2,1,1,s-1)$ we need to verify the same inequalities and since $\frac{rs-2}{r+s-2} \geq 1$ we immediately have
\[
\left(\frac{rs-2}{r+s-2}\right)(r+s-1 - 4)+2 - (r-1+s-2) = (r+s-5)\left(\frac{rs-2}{r+s-2} -1\right) \geq 0.
\]

Hence the lemma is proved for $m=2$. Now assume that the lemma holds for $m-1$. Then 
\begin{align*}
\sum_{i=1}^m P_iQ_i &\leq \left(\frac{rs-2}{r+s-2}\right)\left(\sum_{i=1}^{m-1}{(P_i+Q_i)}-2(m-1)\right)+m-1 + P_mQ_m \\
&=\left(\frac{rs-2}{r+s-2}\right)\left(\sum_{i=1}^{m}{(P_i+Q_i)}-2m\right)+m+(2-P_m-Q_m)\left(\frac{rs-2}{r+s-2}\right)-1 + P_mQ_m. 
\end{align*}
Since $m\geq 2$, $\sum_{i=1}^mP_i \leq r$, and $P_i \geq 1$, we have $P_m \leq r - \sum_{i=1}^{m-1}P_i \leq r-m+1 \leq r-1$ and similarly $Q_m \leq s-1$.  To finish the proof we will show for $1\leq P_m\leq r-1$, $1 \leq Q_m\leq s-1$,
\[
(2-P_m-Q_m)\left(\frac{rs-2}{r+s-2}\right)-1 + P_mQ_m \leq 0.
\]
Again it is sufficient to check for the endpoints $(1,1), (1,s-1), (r-1,1), (r-1,s-1)$. For $(1,1)$ it is trivial. For $(1,s-1)$, and $(r-1,1)$, we only need $\frac{rs-2}{r+s-2} \geq 1$. Finally for $(r-1,s-1)$
\[
(4-r-s)\left(\frac{rs-2}{r+s-2}\right)-1 + (r-1)(s-1) = \frac{-(r-2)^2-(s-2)^2}{r+s-2} \leq 0,
\]
completing the proof.
\end{proof}
\section{Upper bound for $K_{r,s}$ percolation}
For the upper bound we directly appeal to the Proposition 3 from \cite{balogh2012graph}. Let us recall the definition of balanced graph before we state the proposition. 
\begin{dfn}\label{dfnbalanced930} A graph $H$ is called balanced if $e(H) \geq 2 v(H) -2$, and 
\begin{equation}\label{sun302pmapr}
\frac{e(F)-1}{v(F)-2} \leq \lambda(H) : = \frac{e(H)-2}{v(H)-2}
\end{equation}
for every proper subgraph $F \subset H$ with $v(F) \geq 3$.
\end{dfn}
We are now ready to state the proposition that we are going to use to obtain an upper bound.
\begin{prop}\label{sun233pmapril}
If $H$ is a balanced graph then 
\begin{equation}\label{sun239apr}
p_c(n,H) \leq C {\left(\frac{\log n}{\log\log n}\right)}^{2/\lambda(H)} n^{-{1/{\lambda(H)}}}
\end{equation}
for some constant $C=C(H)>0$.
\end{prop}

The following Lemma establishes the upper bound by verifying that $K_{r,s}$ is a balanced graph as long as $r$ is not much larger than $s$.

\begin{lem}\label{lem:sun307pmapr}
For $r\geq s$, $K_{r,s}$ is a balanced graph for $r \geq 4$, $s\geq 3$, and $r\leq (s-2)^2 +s$.
\end{lem}
\begin{proof}[Proof of Lemma \ref{lem:sun307pmapr}] The first task is to verify that $rs\geq 2(r+s)-2$. To check this observe that for $r=4$, and $s=3$ it holds, and $rs - 2(r+s)+2$ is increasing in  both $r$ and $s$ as long as $r,s >2$. Next we verify \eqref{sun302pmapr}. It is easy to see that $\frac{rs-1}{r+s-1}$ and $\frac{rs-1}{r+s-2}$ are increasing function of both $r$ and $s$. Therefore to verify \eqref{sun302pmapr} it is enough to verify for the endpoints, that is
$\frac{rs-2}{r+s-2} \geq \frac{r(s-1)-1}{r+s-1-2}$ and $\frac{rs-2}{r+s-2} \geq \frac{(r-1)s-1}{r-1+s-2}$. The first one is true since $r\geq s$, indeed
\begin{align*}
&\frac{rs-2}{r+s-2} \geq \frac{(r-1)s-1}{r-1+s-2} 
\Leftrightarrow (r-2)^2 +r-s \geq 0.
\end{align*}
The second one is true by our assumption 
\[
\frac{rs-2}{r+s-2} \geq \frac{(r-1)s-1}{r-1+s-2} \Leftrightarrow (s-2)^2 +s-r \geq 0.
\]

\end{proof}

\section{Proof of the general lower bound (Proposition \ref{prop:sun527pmapril})}

In this section we obtain a lower bound for $p_{c}(n,K_{r,s})$ in the unbalanced case. To see this note that for a graph $G$ on $n$ vertices if we have ${\langle G \rangle}_{K_{r,s}} = K_n$, then we also have ${\langle G \rangle}_{K_{r',s'}} = K_n$ for any $r'\leq r$, and $s'\leq s$. Therefore  
$$\mathbb{P}({\langle G_p \rangle}_{K_{r',s'}} \neq K_n) \leq \mathbb{P}({\langle G_p \rangle}_{K_{r,s}} \neq K_n)$$
Now pick $r'(\leq r)$ and $s'(\leq s)$ such that $r'\leq (s'-2)^2 +s'$. Then $\mathbb{P}({\langle G_p \rangle}_{K_{r',s'}} \neq K_n) \rightarrow 1$ as $n \rightarrow \infty$ if $e p n^{1/\lambda(r',s') } (\log{n}) {r's'}\leq {\lambda(r',s')}^2$ (by Lemma \ref{prop:thu217feb}). Therefore $p_{c}(n,K_{r,s}) \geq  {\lambda(r',s')}^2 (e \log n)^{-1} n^{-1/\lambda(r',s') } $. Using this and taking the supremum over all such $r',s'$ we get the lower bound $$(e \log n)^{-1}   \sup_{r'\leq r, s'\leq s, r'\leq (s'-2)^2 +s'}{\lambda(r',s')}^2  n^{-1/\lambda(r',s') }.$$ Finally since $x^2 n^{{-1/x}}$ is increasing in $x$ for $x>0$, we obtain the following supremum is equal to $$(e \log n)^{-1} {\lambda((s-2)^2 +s, s)}^2  n^{-1/\lambda((s-2)^2 +s, s) },$$ 
completing the proof.
\section*{Acknowledgments} We would like to thank the referee for carefully reading the paper and providing valuable suggestions that improved the presentation and readability of the paper. We would also like to thank Neeladri Maitra for helping us with GeoGebra to prepare Figure 1.

\bibliographystyle{plain}
\bibliography{bootstraprefref.bib}

\end{document}